\documentclass{article}
 \usepackage{amsmath,amssymb,amsthm,fullpage}
\title{ Hilbert's Proof\\ 
of His Irreducibility Theorem}

\author{
Mark B. Villarino, 
William Gasarch,
and Kenneth W. Regan}

\date{\today}

\theoremstyle{plain}
\newtheorem{thm}{Theorem}
\newtheorem{prop}[thm]{Proposition}
\newtheorem{lemma}[thm]{Lemma}
\newtheorem{cor}[thm]{Corollary}

\theoremstyle{definition}
\newtheorem{definition}{Definition}

\newcommand{\Erdos}{Erd{\H{o}}s }

\renewcommand{\geq}{\geqslant}  %% (to save typing)
\renewcommand{\leq}{\leqslant}  %% (to save typing)

\newcommand{\hideqed}{\renewcommand{\qed}{}} %% to suppress `\qed'

\newcommand{\Cee}{\mathbb{C}}   %% complex numbers
\newcommand{\Nat}{\mathbb{N}}   %% natural numbers
\newcommand{\Que}{\mathbb{Q}}   %% rational numbers
\newcommand{\Zed}{\mathbb{Z}}   %% integer numbers

\begin{document}

\maketitle

\thispagestyle{empty} %% suppress numbering on page 1

%%KWR531: i restored the abstract, took out my florid use of "heating empty air" which had been for a house-building analogy, and changed what's said about Ramsey Theory since we have the new section.

%%KWR531: The abstract says "Cube Lemma" and the compositor let that stand but wanted lowercase "cube lemma" everywhere in the body.  I think this is actually somehow right: the abstract has the first mention of it and is proclaiming it in some sense. I did keep the usage in the body consistent.  Same thing with "Hilbert's Irreducibility Theorem" and "Ramsey Theory" being left alone here.

\begin{abstract} 
Hilbert's irreducibility theorem is a cornerstone that joins areas of
analysis and number theory. Both the genesis and genius of its proof
involved combining real analysis and combinatorics. We try to expose
the motivations that led Hilbert to this synthesis. 
% Hilbert's famous Cube Lemma supplied fuel for the proof but without the analytical foundation and framework it would have been heating empty air.  
His famous cube lemma anchored the proof but without the analytical foundation and framework it would have had no purpose.  
% The lemma presages Ramsey Theory but we note differences in motivation.
We also assess this lemma as a precursor of Ramsey theory.
\end{abstract}

\section{Introduction.}
\label{sec:intro}

In 1892, David Hilbert 
% published a proof of the following theorem,
%%KWR531: Long ago, I decided not to criticize the above on grounds that Hilbert didn't just publish a proof, because our paper is emphasizing the proof.  But now, saving the one line needed to get "A more modern formulation is:" onto page 1 is THE MAGIC PAGEBREAK SAVER.  (Well OK, we could alternately save 4 words in the paragraph beginning "He began his paper" to achieve the same effect.)
published what is 
known today as \emph{Hilbert's irreducibility theorem}.
%, which he stated as follows.
%%KWR: Since the footnote is not by Hilbert, but we're calling it "his statement", I took it out.
We give his statement, using \emph{integral polynomial} to mean a polynomial in any number of variables whose coefficients are integers.

\begin{thm} % 1
\label{irreducibility}
If $F(x,y,\dots,w;t,r,\dots,q)$ is an irreducible polynomial with 
integral coefficients
% 
% \footnote{A polynomial in any number of variables whose coefficients
% are integers will be called an \emph{integral} polynomial.
% }
% 
in the variables $x,y,\dots,w$ and the parameters $t,r,\dots,q$, then
it is always possible, and indeed in infinitely many ways, to
substitute integers for the parameters $t,r,\dots,q$ such that the
polynomial $F(x,y,\dots,w;t,r,\dots,q)$ becomes an irreducible
polynomial in the variables $x,y,\dots,w$ alone.
\end{thm}

% This is a literal translation from Hilbert's original paper.  It does not meet modern standards.  
%%KWR: "It" seems to refer to the translation.
The statement of Theorem \ref{irreducibility} is a direct translation from \cite{Hilbert1892}.  It falls short of modern precision.    
For example, the irreducibility of $F$ concerns the polynomial in the whole set of variables---parameters included---but the statement is technically false if there are no variables but only parameters. Nor is it clear whether one needs a lot of variables and parameters or whether proving the theorem for one or two of each suffices; this is clarified below.
One of our purposes is to lead readers to appreciate modern rigor and clarity compared to 19th century standards.

% However our purpose is to take the reader from 19th century rigor to modern mathematics so as to appreciate the change.
%%KWR: I feel no paragraph break and no "However" is best for flow.  $...$ around F and period were missing.

To Hilbert, this theorem was not an end in itself but rather a tool to
use for some remarkable applications. A simple one is that if a
polynomial $f(x)$ over $\Zed$ has values that are perfect squares for
all sufficiently large $x$, then $f(x)$ must be the square of some
other polynomial over $\Zed$. One of his most striking results was the following.

%%KWR531: The compositor not only struck out "ones" but also suggested inserting "results".  I also changed "is" to "was" here---we're slicing between past and present anyway, and I felt it sounded better (or rather, I feel it sounds better).  Anyway, this prompted me to be hawkeye on all the edits.

\begin{cor} % 2
For every integer $n\geq 1$ there exist infinitely many polynomials $p$  in
$\Zed[x]$ of degree $n$ such that $p$ has the symmetric group $S_n$ as its Galois
group.
\end{cor}

He began his paper \cite{Hilbert1892} with a statement and proof of
the two-variable case, which is the fundamental step in the proof of
the general theorem. Again in Hilbert's own words, the statement is:

\begin{thm} % 3
\label{irreducibility-1}
If $f(x,t)$ is an irreducible polynomial in the two variables $x$ and
$t$ with integral coefficients
\begin{equation}\label{integral-T}
f(x,t) = T x^n + T_1 x^{n-1} +\cdots+ T_n,
\end{equation}
where $T,T_1,\dots,T_n$ are integral polynomials in $t$, it is always
possible, indeed in infinitely many ways, to substitute an integer for
$t$ in $f(x,t)$ such that the polynomial $f(x,t)$ becomes an
irreducible polynomial of the single variable~$x$.
\end{thm}

%%KWR 906: NEW material here---redirecting the referees' comments at Hilbert.  Then *changed* to use t_1.
\noindent 
This statement also has several issues.  In just eight words of the last sentence, ``$t$'' is first a variable symbol, then part of the name ``$f(x,t)$'' for $f$, and last a constant substituted into $f(x,t)$.  It is hard to avoid the juggling of different meanings that Hilbert expected of his readers, but we will use ``$t_0$'' and ``$t_1$'' to distinguish constants.  Often $t_0$ is a threshold and $t_1 \geq t_0$.  Second, the simpler statement does not clarify what happens when there is no dependence on the variable $x$.  The set $\Zed[x]$ of polynomials $f$ in one variable $x$ having coefficients in $\Zed$ includes the constant polynomials $f(x) = 0$, $f(x) = 1$, $f(x) = -1$, and so on.  Likewise, $f(x,t) = t^2 + t + 2$ counts as a member of $\Zed[x,t]$.  It satisfies the hypotheses because it is irreducible and because the choices $T = T_1 = \dots = T_{n-1} = 0$ and $T_n = t^2 + t + 2$ count as ``integral polynomials in $t$.''  But the conclusion is false because for every integer $t_1$, the \emph{constant} $t_1^2 + t_1 + 2$ is an even number, so the constant \emph{polynomial} $f'(x) = f(x,t_1) = t_1^2 + t_1 + 2$ is reducible as a member of $\Zed[x]$.\footnote{
If this objection seems trivializing, consider instead the polynomial $g(t) =  t^2 + 1$, which is likewise irreducible.  Whether $g(t)$ takes on infinitely many prime values is one of four problems presented by Edmund Landau to the 1912 International Congress of Mathematicians---all still unsolved---and it must have been outside the boundary of what Hilbert's words were meant to embrace in 1892.  Thus Hilbert's theorem borders on matters of lasting depth.
} %% end of footnote %%KWR906: One can delete my short last sentence here.

We will see at the start of Section~\ref{sec:pseries} how the proof needs $f(x,t_1)$ to have at least one (complex) root for all but finitely many $t_1$.  For this it suffices
%there to be infinitely many $t_1$ such that $f(x,t_1)$ is unbounded in $x$. 
%%KWR906: Rewritten per Bill's latest comments.  Is my line above any weaker? sufficient?  See also my long comment at the head of the Puiseux section.  I've tried to tiptoe around a full formal definition of "degree of x in a two-variable polynomial f(x,t)" and also avoid the painful change of letters Bill mentioned.
that the two-variable polynomial $f(x,t)$ is not independent of $x$, so that $x$ has degree at least $1$.  The crispest modern way we know to say this is ``$f \in \Zed[x,t] \setminus \Zed[t]$,'' where $\setminus$ means difference of sets.  This yields the following statement.

%A more modern formulation is:

\begin{thm} % 4
\label{irreducibility-2}
%Let $f(x,y) \in \Zed[x,y]$ in which the degree in $x$ is at least $1$ be irreducible. For an infinite number of $t \in \Zed$, $f(x,t)$, as an element of $\Zed[x]$, is irreducible.
Let $f(x,y) \in \Zed[x,y] \setminus \Zed[y]$. For an infinite number of $t_1 \in \Zed$, $f(x,t_1)$, as an element of $\Zed[x]$, is irreducible.
\end{thm}

\noindent 
In fact, Hilbert proved the \emph{contrapositive}, which can be
formulated as follows.

\begin{thm} % 5
\label{irreducibility-contra}
Let $f(x,y)\in \Zed[x,y] \setminus \Zed[y]$. If there exists $t_0$ such that for every
integer 
%% $t^* \ge t_0$, $f(x,t^*)$ is reducible over $\Zed[x]$, then $f(x,y)$ is reducible over~$\Zed[x,y]$.
%%KWR906: t^* is only used again in the "goal" theorem statement, so I shifted to t_1.
%% That theorem incidentally used "reducible in", so "over" here was definitely inconsistent.  I've used "in".
$t_1 \ge t_0$, $f(x,t_1)$ is reducible in $\Zed[x]$, then $f(x,y)$ is reducible in $\Zed[x,y]$.
\end{thm}

%BILL'S COMMENT: THE REST OF THIS SECTION I REWROTE. IT HAS MORE OF A SKETCH OF WHATS TO COME
%THEN YOU HAD PREVIOUSLY.

Hilbert proved Theorem~\ref{irreducibility-contra} by formulating what is today called
\emph{Hilbert's cube lemma}. It can be viewed not only as an enhanced
form of Dirichlet's pigeonhole principle but also as the first
statement of a Ramsey-type theorem. 

In Section~\ref{sec:ramsey} we discuss Ramsey theory to illustrate why
Hilbert's cube lemma is regarded as 
% being in 
belonging to that field.
In Section~\ref{sec:cube-lemma} we state and give a simple modern proof of the Hilbert's cube lemma and describe optimizations
(we discuss Hilbert's original proof in Section~\ref{sec:conc}).
It is easy to appraise the Hilbert cube lemma as a gem in an isolated setting and
forget the quest that led to it, which was was \emph{to find a
polynomial factor, $\varphi(x,t) \in \Zed[x,t]$, of the polynomial
$f(x,t)$.}

In Sections~\ref{sec:monic} through~\ref{sec:int} 
we provide a motivated account of Hilbert's beautiful proof of (ir)reducibility by
putting ourselves in his shoes and following the trail of ideas that
we find in his 1892 paper. We have tried to make it as self-contained and elementary as possible.

After Hilbert, many mathematicians offered other proofs of the
irreducibility theorem. Many of these
proofs use so-called ``density" arguments, a standard technique in
today's Diophantine approximation theory, but a far cry from the
natural idea of Hilbert to find a factor of a reducible polynomial. We
will say more about modern proofs in Section~\ref{sec:later}.

Hilbert remains one of the greatest mathematicians of all time. His
original proof still contains insights and arguments that are well
worth study even today. We offer the reader a detailed exposition of
this proof in hope of saving it from the oblivion of history.

%%%%%%%%%%%%%%%%%%%%%%%%%%%%%%%%%%%%%%%%%%%

%BILL COMMENT: THIS SECTION IS NEW
%%KWR: I've given it a full tightening pass.  This includes:
% () referring to references [10,16,21] just once.
% () giving the firstname on first mention of a principal figure
% () not "proclaiming" definitions that are used only here.

\section{Ramsey Theory.}
\label{sec:ramsey}

%%KWR531: Note---even in the first line below, the compositor would seem to want us to lowercase the "t".

Theorems in Ramsey theory 
% are almost always of the form (informally):
almost always follow this informally-stated pattern:

\begin{center}
\fbox{\parbox{6.5cm}{\it For any coloring of a large enough object
% (e.g., an initial segment of the natural numbers)  
there is a nice monochromatic sub-object.}
% (e.g., a monochromatic arithmetic progression).}
}
\end{center}

We give three examples of such theorems along with some of the history. 
% For more of the history of these theorems see Alexander Soifer's book~\cite{coloring}.
See \cite{GRS,RamseyInts,ramseypromel} for more on these theorems and also Alexander Soifer's book~\cite{coloring} for more of the history.

In 1916, Issai Schur~\cite{Schur} 
% (see also~\cite{GRS,RamseyInts,ramseypromel}) 
proved the following statement.  
A \emph{$c$-coloring} of a set $S$ is formally a map from $S$ to $\{1,2,\dots,c\}$.

%%KWR906: Added last sentence rather than adopy referee's "...with $c$ colors" because we use it again before graphs are introduced.

\begin{lemma} 
\label{schur}
For all $c$ there exists $S=S(c)$ such that for all $c$-colorings of $\{1,\ldots,S\}$ 
there exists a  monochromatic triple $x,y,z$ such that $x+y=z$.
\end{lemma}

%%KWR531: Above, the compositor wanted us not to indent the lines between the theorem statements---IMHO because they collectively with those statements form one logical paragraph.  Whereas here it strikes me that this one starts a new paragraph and the Ramsey one forms a paragraph unto itself.  But maybe worth your seeing if you agree.

Schur viewed his lemma as a means to an end and so
did not launch what is now called Ramsey theory. 
He used it to prove the following theorem in number theory.

%%KWR531: It seems we use colons before theorem statements more often than not---and the compositor let them stands.  I've left a period in only when it ends a clause that does not include "following" or similar.  

%\begin{definition}
%If $p$ is a prime then let $\Zed_p^*$ be the numbers
%$\{1,\ldots,p-1\}$ together with the operation of 
%modular multiplication. This is known to be a group.
%\end{definition}
%%KWR: Moved to body of proof.  This proclaimed definition interposed itself before the "theorem" anyway.

\begin{thm} 
\label{schurfermat}
Let $n\ge 1$. 
There exists $q$ such that, for all primes $p\ge q$,
there exists $x,y,z\in \{1,\ldots,p-1\}$ such that 
$x^n+y^n\equiv z^n \pmod p$.
\end{thm}

\begin{proof}
Given $n$ let $q=S(n)$.
Let $p$ be a prime such that $p\ge S(n)$.
Then $\Zed_p^*$, which denotes the numbers
$\{1,\ldots,p-1\}$ together with the operation of 
 multiplication modulo $p$, forms a group.
All arithmetic henceforth is in $\Zed_p^*$.

Let $H = \{ x^n  \mid x\in \Zed_p^* \}$. Clearly $H$ is a subgroup of $\Zed_p^*$.
It is known that $|H|=\frac{p-1}{\gcd(n,p-1)}$ so the number of cosets is 
$c=\frac{p-1}{|H|} = \gcd(n,p-1) \le n$.
We denote the cosets by
$d_1H,\ldots,d_cH$. 

Consider the following $c$-coloring of $\{1,\ldots,p-1\}$:
color $x$ by $i$ such that $x\in d_iH$.
Since $c\le n$ and $p-1\ge S(n)$, by Schur's lemma,
there exists a monochromatic $x_1,y_1,z_1$ such that $x_1+y_1=z_1$.
Since they are all in the same coset, there exists $d$ such that
$x_1,y_1,z_1 \in dH$. Hence 
$x_1=dx^n$, 
$y_1=dy^n$, 
$z_1=dz^n$.
Since $dx^n + dy^n = dz^n$ we get $x^n+y^n=z^n$.
\end{proof}

% Theorem~\ref{schurfermat} is of interest since it refuted the following plan for proving
% Fermat's last theorem: Show that for all $n\ge 3$ there exists $p$ such that $x^n+y^n\equiv z^n$
% has no solution with $x,y,z\in \{1,\ldots,p-1\}$.
%%KWR: Surely not what you mean to say, since I could just pick p=3?

Theorem~\ref{schurfermat} refuted the idea of proving
Fermat's last theorem by showing that for all $n\ge 3$ there are arbitrarily large $p$ such that $x^n+y^n\equiv z^n$
has no solution modulo $p$.

In 1927, Bartel van der Waerden~\cite{VDW} 
%(see also~\cite{GRS,RamseyInts,ramseypromel}) 
proved the following theorem which now bears his name.

\begin{thm} 
\label{vdw}
For all $k,c$ there is a number $W=W(k,c)$
such that, for all $c$-colorings of $\{1,\ldots,W\}$,
there exists a monochromatic arithmetic sequence of length $k$.
\end{thm}

% Van der Waerden proved this theorem because Baudet had conjectured that it was true.
The title of \cite{VDW} credits Pierre Baudet with having conjectured this, but
Soifer~\cite{coloring} gives evidence that Schur had also done so.
Even though van der Waerden did not have another goal in mind, he did not pursue
this line of research and 
% hence he 
so did not launch what is now called Ramsey theory.

Frank Ramsey~\cite{ramsey30} 
% (see also~\cite{GRS,RamseyInts,ramseypromel}) 
proved the following theorem that now bears his name.  
% (we only state the case for graphs, not hypergraphs):
As others often do, we state only the case for graphs, not hypergraphs.
%
%\begin{definition}~
%\begin{enumerate}
%\item
%A {\it graph} is an ordered pair $(V,E)$ where $E$ is a subset of unordered pairs of elements from $V$.
%The set $V$ is called {\it the vertices} and the set $E$ is called {\it the edges}.
%To picture a graph imagine a set of points in the plane (the vertices) and connect any
%two $x,y$ if $\{x,y\}\in E$.
%\item
%Let $m\in \Nat$. {\it The complete graph on $m$ vertices}, denoted $K_m$,
%has $V=\{1,\ldots,m\}$ and $E$ is the set of all pairs of elements from $V$.
%\item
%Let $R\in\Nat$. If the edges of $K_R$ are colored then a {\it monochromatic $K_m$}
%is a set of $m$ vertices such that all of the edges between them are the same color.
%\end{enumerate}
%\end{definition}
%
%%KWR906: I changed the following per referee comment.  I shied away from Bill's suggestion 
%% ---2) Page 3. Referee is right that E is defined as a set of ordered pairs.  Fix: a set E \subseteq \binom{V}{2} (the set of all unordered pairs of elements of $V$)---
%% because I feared defining \binom{V}{2} as a set would be a confusing novelty.
%% BILL911 COMMENT : $\binom{V}{2}$ is becoming standard.
A \emph{graph} consists of a set $V$ of \emph{vertices} and a set 
$E \subseteq \binom{V}{2}$ (the set of unordered pairs of elements of $V$).
The graph is \emph{complete} if $E$ includes all such pairs and is then denoted by $K_n$, where $n = |V|$.  
%%A \emph{$c$-coloring} of the edges is a mapping $f$ from $E$ to $\{1,\dots,c\}$.  
%%KWR906: Changed since defined above now.
We consider $c$-colorings of the \emph{edges}, namely mappings $f$ from $E$ to $\{1,\dots,c\}$.
A \emph{monochromatic $K_m$} means a subset $V' \subseteq V$ of size $m$ and $c' \leq c$ such that for all distinct $u,v \in V'$, $(u,v)$ is an edge and $f(u,v) = c'$.

%%KWR906: I used Bill's (second) suggestion.
\begin{thm} 
\label{ramsey}
For all $c,m$ there exists a number $R$
such that for all $c$-colorings of the edges of $K_R$
there exists a monochromatic $K_m$.
\end{thm}

The least such $R$ satisfying the conditions of Theorem \ref{ramsey} is denoted by $R_c(m)$.  The folkloric example of this theorem is 
% the problem of proving 
that in any group of six people, \emph{at least three know each other or at least three are complete strangers}.  If the six are the vertices of a $K_6$ and each edge is 
% given one of two colors 
colored green or blue (friends or strangers), then the theorem says there is at least one monochromatic triangle. In fact, there are at least \emph{two} such triangles, whereas $K_5$ has none when a green five-pointed star is inscribed in a blue pentagon, so that $R_2(3) = 6$.

%%KWR: Added pentagon reference so as to exemplify the R(c,m) notation.

% Ramsey viewed his lemma as a means to an end and 
%%KWR531: Added mention of logic per WG.
% hence he   %%KWR531: Another reason I changed "hence" to "so" is that Ramsey died young.
Ramsey applied his lemma to problems in mathematical logic.  
He viewed it as a means to an end and
so did not launch what is now called Ramsey theory. 

In 1892, 
% (see also~\cite{GRS,RamseyInts,ramseypromel}) ,
before all of the results above,  Hilbert~\cite{Hilbert1892} proved 
% the Hilbert cube lemma
the lemma featured in the next section.
% which we state and prove in the next section. When you read it you will see that it is very much like the three statements in this section as it involves
Like the three statements above (Lemma 6, Theorems 8 and 9), it applies to
any $c$-coloring and yields a monochromatic nice substructure.
Hilbert viewed his lemma as a means to an end and 
% hence he 
so did not launch what is now called Ramsey theory. He used it to prove the
Hilbert irreducibility theorem, which is
%, as you know, the topic of this paper.
our main topic.

Who did launch Ramsey theory? 
% \Erdos claims that Szekeres rediscovered the statement
% and proof of Ramsey's Theorem while proving the following~\cite{ES}
%%KWR: Is Erdo"s still publishing?  I know Trump would say he's doing even better than Frederick Douglass. 
%%KWR: It's also not completely clear to me whether [ES] is the reference for the theorem or for just Erdos's claim?  OK, I found it http://www.numdam.org/article/CM_1935__2__463_0.pdf and checked---there is no claim in the paper so I trust my words get it right:
%
Speaking about his joint paper in 1935 with George Szekeres \cite{ES}, Paul \Erdos said that it was Szekeres 
who rediscovered the statement and proof of Ramsey's theorem.  They used it as a means to the following end.

\begin{thm}
For all $n\ge 3$ there exists $m > n$ such that for any $m$ points in the plane
in general position there exists $n$ points that form a convex hull.
\end{thm}

%%KWR: Added---I feel we need some kind of section closing like it.
\noindent
But they also 
% centered 
%%KWR531: Agreed with WG---a clique has no "center" anyway.  I really want to say "anchored" but that's also mixing metaphors.  Could say "generated" but "attracted" is simpler.
attracted a clique of mostly Hungarians who developed the ideas, conjectures, and results that grew into Ramsey theory as we know it.

%%KWR: I have not proofread beyond here yet.

%%%%%%%%%%%%%%%%%%%%%%%%%%%%%%%%%%
\section{The Cube Lemma.}
\label{sec:cube-lemma}

Hilbert's first paragraph in \cite{Hilbert1892} crisply framed the \emph{problem} of
irreducibility under substitutions represented by the statement of
Theorem~\ref{irreducibility}. Then he continued right away: ``Our
developments rest on the following lemma.'' We reproduce his words but
change his $a,\mu$ to $c,\beta$ and compress his displayed formulas
using variables $b_1,\dots,b_m$ that take only the values $0$ or $1$:

%\begin{quotation}
%Given an infinite integer sequence $a_1,a_2,a_3,\dots$ in which
%generally each $a_s$ denotes one of the $c$-many positive integers
%$1,2,\dots,c$, let $m$ be any positive whole number. Then there are
%always $m$-many positive whole numbers
%$\mu^{(1)},\mu^{(2)},\dots,\mu^{(m)}$ such that the $2^m$ elements
%\[
%a_{\beta + \sum_{i=1}^m b_i \mu^{(i)}}
%\]
%for infinitely many whole numbers $\beta$ are collectively the same
%number $G$, where $G$ is one of the numbers $1,2,\dots,c$.
%\end{quotation}

%%KWR531: The compositor didn't like indenting this, despite its being standard LaTeX behavior for this environment.  I wonder if it will fly better---and look better---as a boxed statement.  I had boxed statements to go with boxed equation sets which the compositor let stand, and Mark added one for Hilbert's key insight.  To discuss maybe...

\begin{center}
\fbox{\parbox{11cm}{Given an infinite integer sequence $a_1,a_2,a_3,\dots$ in which
generally each $a_s$ denotes one of the $c$-many positive integers
$1,2,\dots,c$, let $m$ be any positive integer. Then there are
always $m$-many positive integers
$\mu^{(1)},\mu^{(2)},\dots,\mu^{(m)}$ such that the $2^m$ elements
\[
a_{\beta + \sum_{i=1}^m b_i \mu^{(i)}}
\]
for infinitely many integers $\beta$ are collectively the same
number $G$, where $G$ is one of the numbers $1,2,\dots,c$.}}
\end{center}

\noindent
Call those elements collectively the \emph{$m$-cube}, which we can
denote by $C(\beta;\mu_1,\dots,\mu_m)$. The sequence
$a_1,a_2,a_3,\dots$ can be called a \emph{coloring} of $\Nat^+$ using
$c$~colors. Thus the conclusion is that every coloring gives rise to
\emph{increments} $\mu^{(1)},\mu^{(2)},\dots,\mu^{(m)}$ that yield a
monochromatic $m$-cube for infinitely many starting points $\beta$.
This is implied by the following finitistic statement, which we regard
as \textbf{Hilbert's cube lemma} in the modern sense.

\begin{lemma} 
\label{le:cube}
For all $m,c$ there is a number $H$ such that, for all $c$-colorings
of $\Nat^+$ and all intervals of length $H$ in $\Nat^+$, there is a
monochromatic $m$-cube within the interval.
\end{lemma}

%BILL COMMENT: REMOVE these two lines as they are already in the intro:

%The following proof tries to be simplest; we discuss optimizations and
%Hilbert's original proof in Section~\ref{sec:conc} at the end.

%BILL911. I define $H_m$.

\begin{proof}
We fix $c$. We will let $H_m$ be a value of $H$ that satisfies the theorem.
We prove, by induction on $m$, that $H_m$ always exists.
For the base case $m = 1$, we can
take $H_1 = c + 1$. This just says that for any $c$-coloring of an
interval of length $c + 1$ there will be two elements that are the
same color. Taking $\beta$ to be the smaller one and $\beta + \mu_1$
the larger one, $C(\beta;\mu_1)$ is a monochromatic 1-cube.

For the induction step, assume that $h = H_{m-1}$ exists. We show that, for
any $c$-coloring of an interval of length 
$H_m = h \!\cdot\! (1 + c^h)$, there is a monochromatic $m$-cube. Let
$COL$ be a $c$-coloring of an interval of length $H_m$. Partition the
interval into $1 + c^h$ blocks of size $h$. By the pigeonhole
principle, some two of those blocks have the same sequence of $h$
colors. By the induction hypothesis, the former has a monochromatic
$(m-1)$-cube $C(\beta;\mu_1,\dots,\mu_{m-1})$, and since the color
sequence of the latter is the same, it has
$C(\beta';\mu_1,\dots,\mu_{m-1})$ with the same color and increments
but $\beta' > \beta$. Take $\mu_m = \beta' - \beta$. Then
$C(\beta;\mu_1,\dots,\mu_m)$ is the required monochromatic $m$-cube.
This proves the lemma statement with $H = H_m$.
\end{proof}

%%KWR906: Inserted per Bill's suggestion.  Are the notation H(m,c) and term 
%% "Hilbert Cube Number" standard?  My sense was only quasi-so.
%%KWR906 later---Bill, is my added last line in the proof good enough?

By analogy with Ramsey numbers, one can denote the least such $H$ by $H(m,c)$ 
and call it a ``Hilbert Cube Number.''  
The above proof embodies a recursive upper bound
$H(m,c) \leq H(m-1,c)(1 + c^{H(m-1,c)})$, with basis $H(1,c) = 1$ for
all $c$. This is far from best possible. When
$2 \leq m \leq c$, one can improve the upper bound to 
$H(m,c) \leq h(1 + c(m-1)^{h})$, where $h = H(m-1,c)$, by a different
counting argument. One can further tweak this by using $\binom{m-1}{h}$ in
place of $(m-1)^h$. These formulas are not bounded by any fixed tower
of exponents in $c$ and~$m$.

As observed by Brown et~al.~\cite{BCEG}, Hilbert's original proof
yields bounds with $(c + 1)$ rather than $(m - 1)$ in the base and the
Fibonacci number $F_{2m}$ in the exponent, that is,
$H(m,c) \leq (c + 1)^{F_{2m}}$, where $F_0 = 0$, $F_1 = 1$, $F_2 = 1$,
$F_3 = 2$, and so on. These bounds have doubly-exponential growth.
Szemer\'edi \cite{Sz} (see also~\cite{GRS}) improved both the bounds
and the nature of the result.  Here is his more modern form of the lemma:

\begin{lemma}
Let $H,c > 0$ and let $A$ be a subset of
$[1,\dots,H]$ such that $|A| \geq H/c$.  Then for some constant $C$ depending
only on $c$, $A$ contains an
$m$-cube where $m \geq \log\log(H) - C$.
\end{lemma}

%%KWR906: I moved most of the text Bill suggested, but I'm keeping
%%the following in section 13 where it serves a context there.

%The best known upper and lower bounds appear still to be those of
%Gunderson and R\"odl~\cite{GR}:
%\[
%c^{(1 - \epsilon_c)(2^m - 1)/m} \leq H(m,c) \leq (2c)^{2^{m-1}},
%\]
%where $\epsilon_c \to 0$ as $c \to \infty$. The same upper bound was
%recently ascribed to \cite{GRS2} by Conlon, Fox, and Sudakov
%\cite{CFS} but see also S\'andor~\cite{sandor} with different
%asymptotics. Erd\H{o}s and Tur\'an~\cite{et} proved that $H(2,c)$
%is asymptotic to $c^2$, but the remarks in \cite{BCEG} that less is
%known about $H(m,c)$ for fixed $m \geq 3$ appear still in force, and
%\cite{CFS} remarks that $H(m,2)$ depends on unknown properties of van
%der Waerden numbers.
%BILL COMMENT: REMOVE  the following two paragraph since it is already in the introduction.

%%%%%%%%%%%%%%%%%%%%%%%%%%%%%%%%%%%%%%%%%%%%%%%%%%%%%%%%%%%%%%%%%

\section{Monic polynomials.}
\label{sec:monic}
 
 %%KWR531: Inserting a theorem statement as the goal not only meets a comment by Referee #2 but makes the whole flow much clearer to me. I postpone the substitution x = y/T until the proof of the implication.  Inserting the Ramsey section also made it more important IMHO to have a theorem statement here as the goal.
 %
 % I could with a few more words change Theorem~\ref{thm:goal} to say:
 % "Let $g(y,t) \in \Zed[y,t]$ be monic for any fixed value of $t$.  If there exists $t_0$ such that for all $t > t_0$, $g(y,t)$ is reducible in $\Zed[y]$, then $g(y,t)$ is reducible in $\Que[y,t]$."
 
Hilbert begins by reducing his general problem to the case of
\emph{monic polynomials in one variable $x$} with \emph{rational}
coefficients. That is, he shows that the following statement suffices to prove
Theorem~\ref{irreducibility-contra}.

\begin{thm}\label{thm:goal}  %%KWR: Now 12
% Let $g(y,t) \in \Zed[y,t]$ which is of degree $\geq 1$ in $y$.  If there exists $t_0$ such that for all integers $t ^*> t_0$, $g(y,t^*)$ is monic and reducible in $\Zed[y]$, then $g(y,t)$ is reducible in $\Que[y,t]$.
%%KWR906: Per several things noted above, some changes:
Let $g(y,t) \in \Zed[y,t] \setminus \Zed[t]$.  If there exists $t_0$ such that for all integers $t_1 \geq t_0$, $g(y,t_1)$ is monic and reducible in $\Zed[y]$, then $g(y,t)$ is reducible in $\Que[y,t]$.
\end{thm}

\noindent
Proving Theorem~\ref{thm:goal} will occupy the upcoming sections, but here we show the following.

\begin{prop} % 7  %%KWR: Now 13
\label{prop:Que2Zed}
% If $g(y,t)$ is reducible over $\Que[y,t]$, then $f(x,t)$ is reducible over $\Zed[x,t]$.
Theorem~\ref{thm:goal} implies Theorem~\ref{irreducibility-contra}.
\end{prop}

%\noindent
The following transformation is used not only to prove the implication but also to motivate the machinery for proving Theorem~\ref{thm:goal}.
Recall from \eqref{integral-T} the integral polynomials $T,T_1,\dots,T_n$ in $t$ such that
$f(x,t) = Tx^n + T_1 x^{n-1} + \cdots +T_{n-1} x + T_n.$
Note that $T$ and the $T_j$'s become integer constants for any fixed value of $t$.  Define
\begin{equation}\label{gyt}
g(y,t) = y^n + S_1 y^{n-1} + \cdots + S_{n-1} y + S_n,
\end{equation}
where for each $j$, $1 \leq j \leq n$, $S_j = T_j T^{n-j}$.  Then
\[
g(y,t) = T^{n-1} f(\frac{y}{T},t).
\]
Thus $g(y,t)$ is defined by rational transformation of the argument $x$ in $f(x,t)$
but still comes out integral.  To work with this, we preface four observations,
of which the first is a famous result of Gauss called his \emph{polynomial lemma}.

%%KWR602: My restructuring of the lemma to cover comments by Referee #3 as well as #2 also gives (b) and (c) more incremental statements.  BTW, I caught a missing T^{n-1} in the formula just above---but please double-check it and the new lemma version.
%%KWR906: Notice we are saying "factors" rather than "reduces", but IMHO that's appropriate in a PID.

\begin{lemma} % 8  %%KWR: now 14
\label{lem:factors}
\begin{enumerate}
\item % (a)
If a monic polynomial in $\Zed[y]$ factors in $\Que[y]$, then it factors in $\Zed[y]$.
\item % now (b)
A polynomial $\psi(y)$ divides a polynomial $f \in \Zed[x,y]$ if and only if, upon writing
\[
f(x,y) = a_0(y)x^{n} + a_1(y)x^{n-1} +\cdots+ a_{n-1}(y)x + a_n(y),
\]
we have that $\psi(y)$ is a factor of each $a_j(y)$.
\item % now (c)
If $\psi(y)$ is irreducible and divides $f\cdot g$, where $g \in \Zed[x,y]$ is written as
\[
g(x,y) = b_0(y)x^{n} + b_1(y)x^{n-1} +\cdots+ b_{n-1}(y)x + b_n(y),
\]
then either $\psi(y)$ is a factor of all $a_i(y)$ or it is a factor of all $b_i(y)$.
\item % now (d)
If $f(x,y)$ can be factored into the product of two polynomials in $x$
whose coefficients are rational functions of $y$ with integral
coefficients, i.e., belong to $\Que(y)[x]$, then it can be factored into the
product of two polynomials in $\Zed[x,y]$.
\end{enumerate}
\end{lemma}

\begin{proof}
Part (a) can be understood as saying the product of two monic polynomials, each with at least one
rational noninteger coefficient, must have at least one rational noninteger coefficient.  The 
intuition for (b) and (c) is that since $x$ occurs nowhere else, it cannot help
$\psi(y)$ divide $f$ or $f\cdot g$ any other way than stated; proofs may be found in
B\^ocher \cite[pp.~203--204]{Boc}.
%has a formal proof.
To prove
(d), we write the given factorization in the form
\[
f(x,y) 
= \frac{f_1(x,y)}{\varphi_1(y)}\cdot \frac{f_2(x,y)}{\varphi_2(y)},
\]
where $f_1(x,y)$, $f_2(x,y)$, $\varphi_1(y)$ and $\varphi_2(y)$ are
integral polynomials such that $f_1$ is not divisible by any factor of
$\varphi_1(y)$ and $f_2$ is not divisible by any factor of
$\varphi_2(y)$. By part~(c), since $f_1 \cdot f_2$ is divisible by
$\varphi_1 \cdot \varphi_2$, $f_1$ has the complete polynomial
$\varphi_2$ as a factor and $f_2$ has the complete polynomial
$\varphi_1$ as a factor. By~(b) we can cancel $\varphi_2$ from the
coefficients of $f_1$ and we can cancel $\varphi_1$ from the
coefficients of~$f_2$. This gives us our factorization into two
polynomials in $\Zed[x,y]$.
\end{proof}

\begin{proof}
[Proof of Proposition~\ref{prop:Que2Zed}]   %%KWR531: Added "of Proposition 12"
%%KWR906: Here too the referee could kvetch that t is a constant in this context but we are treating it like a variable.  So I prepended a sentence and kept to my scheme of using both t_1 and t_0, the latter for asymptotics.
Let $t_1 \geq t_0$ in the hypothesis of Theorem~\ref{thm:goal}.
Recall that $g(y,t) = f(\frac{y}{T},t)$ for any $t$.  
%%KWR906: Use t_1 in the above line too?
Since $f(x,t_1)$ factors in $\Zed[x]$, $g(y,t_1)$ factors in $\Que[y]$.  But since $g(y,t_1)$ is an integral polynomial
and is monic in the one variable $y$, it factors in $\Zed[y]$ by Gauss's polynomial lemma.

Thus we have satisfied the hypothesis of Theorem~\ref{thm:goal}---and with the same $t_0$ as in
Theorem~\ref{irreducibility-contra}.  Assuming its conclusion gives us
\[
g(y,t) = \Psi(y,t) \Psi'(y,t),
\]
%where $\Psi(y,t) \in \Que[x,t]$ and $\Psi'(y,t) \in \Que[x,t]$. where $\Psi(y,t)$ and $\Psi'(y,t)$ belong to $\Que[y,t]$. 
%%KWR531: The AMM compositor as well as we completely missed that this line was garbled.
%
where $\Psi(y,t)$ and $\Psi'(y,t)$ belong to $\Que[y,t]$.  Substituting back $y = xT$ yields the following equation for our
original polynomial:
\[
f(x,t) = \frac{\Phi(x,t) \Phi'(x,t)}{AT^{n-1}},
\]
where $\Phi(x,t)$ and $\Phi'(x,t)$ both belong to $\Zed[x,t]$, 
$A \in \Zed$, and $T \in \Zed[t]$ . Now part~(d) of our lemma completes
the proof.
\end{proof}

%%KWR531: Inserted "now"

% The whole proof of Theorem~\ref{irreducibility-contra} (and hence Theorems~\ref{irreducibility-1} and~\ref{irreducibility}) now boils down to showing that \emph{if for almost all $t \in \Zed$, $g(x,t)$ is reducible over $\Zed[x]$, then $g(x,y)$ is reducible over $\Que[x,y]$.}

%Proposition~\ref{prop:Que2Zed} is a two-variable variant of an important and famous result that algebraists call \emph{Gauss's Lemma for Polynomials.}  The one-variable case that appears in paragraph 42 of Gauss's \emph{Disquisitiones} \cite{Gauss} is as follows:
%
%\emph{The product of two monic polynomials over the rationals with at least one rational non-integral coefficient is itself a monic polynomial over the rationals with at least one rational non-integral coefficient.}
%
%The modern version follows at a stroke upon observing that if even one of its factors has a non-integral coefficient, so does the original reducible integral polynomial, a contradiction:
%
%\emph{If a monic integral polynomial is reducible over the rational numbers, then its polynomial factors are themselves integral polynomials.}
%
%We know of  no occurrences of Gauss's version in literature outside the \emph{Disquisitiones}, yet all proofs of the modern statement are modeled on Gauss's original proof, as here.  
Gauss used his lemma, which appeared on page 42 of his \emph{Disquisitiones} \cite{Gauss}, to give the first proof of the irreducibility of the cyclotomic polynomial of prime degree over the rationals.  As we've seen, Hilbert used it to reduce the irreducibility theorem to the case of monic polynomials with rational coefficients.  
But to go further and prove Theorem~\ref{thm:goal}, a new tool is needed.

%%KWR531: Added last sentence as a transition.

%%%%%%%%%%%%%%%%%%%%%%%%%%%%%%%%%%%%%%%%%%%%%%%%%%%%%%%%%%%%%%%%%

\section{Puiseux series.}
\label{sec:pseries}

%%KWR906: The problem with implementing Bill's suggestion "Also- g is of degree n, this should be mentioned." is first that we need to specify "y has degree n in g" and second, we haven't defined this and it's a little slimy to define.  What's being left out is the idea that we only need properties as t gets large (does the "at infinity" channel this?)---recall my "weaker" line which I commented out in the intro.  This is also touching Bill's "painful" comment 7:
%
%7) Referee 3 has a point and I have a partial solution but it might be
%painful.
%
%Note that following
%
%Theorem 4 is f(x,y) where x,y are variables, and you look at an
%infinite number of t\in Z
%
%Theorem 12 is g(y,t) where y,t are variables and you also look at
%an infinite number of t\in Z
%
%I would suggest that we change Theorem 12 (and, alas, much of the
%rest o the paper) to g(x,y) and look at infinitely many t in Z.

% The fundamental theorem of algebra shows us that 
% the equation $g(y,t) = 0$, where $y$ has degree $n$ in $g$, has $n$ complex roots for each value of $t$. 
Let $n$ be the maximum degree of $g(y,t_1)$ as a polynomial in $y$ as $t_1$ varies.  It is possible that $g(y,t_1)$ has degree less than $n$ for some values $t_1$ owing to cancellations, but these values are isolated and we will effectively be able to ignore them.  
%The fundamental theorem of algebra shows that the equation $g(y,t_1) = 0$ (for any values of $t_1$) has $n$ complex roots counting multiplicities.
The fundamental theorem of algebra shows that the equation $g(y,t_1) = 0$ (for other values of $t_1$) has $n$ complex roots.
% BILL911- the sentence was `polynomial in $x$ as $t_1$ varies' and I made it
% `polynomial in $y$ as $t_1$ varies'
% BILL911- I changed `from other values of $t_1$' to `from any values of $t_1$'
Thus, thinking informally, we can postulate $n$ functions of $t$, say $y_1(t),\dots,y_n(t)$,
that satisfy the equations. Hilbert brought this into reality by using a refined form of the
\emph{implicit function theorem}, which we refer to as \emph{Puiseux's
theorem} (see discussion of origins below). It says that the $n$ root
functions $y_1(t),\dots,y_n(t)$ can be expressed in a concrete way by
means of fractional power series in decreasing powers of the variable.
These are the so-called \emph{Puiseux series at infinity}, defined as follows.

\begin{definition} % 1
\label{de:pu}
%Let $m \in \Nat$. 
%%KWR906: I used Bill's suggestion to elide m, but I felt writing "u(x)" might be confusing when the reader sees u(x^{1/k}). Since we have so many restrictions to coeffs in Z or Q, I chose to emphasize the coeffs can be complex while circumlocuting away from "X[x]".   Not only don't we refer to m this way, it could confuse with the cube-lemma m.
A Puiseux series at infinity is an expression of the
form
\[
u(x^{1/k}) + \sum_{i=1}^\infty \frac{\mathrm{B}_i}{x^{i/k}},
\]
%%where $u(x) \in \Cee[x]$ is of degree $m$, $k \in \Nat^+$, and
where $u$ is a polynomial with possibly complex coefficients, $k$ is a positive integer, and
$\mathrm{B}_1,\mathrm{B}_2,\ldots \in \Cee$.
\end{definition}

\noindent
We adopt the next theorem statement from \cite[pp.~80--81]{Tz}
with slight alterations in notation and formatting. Its power series
are called \emph{Puiseux expansions at infinity}.

\begin{thm}[Puiseux's Theorem] % 9
Given $g(y,t)$ as above, there are $n$
distinct  power series
\begin{equation}
\label{Puiseux1} % (1)
\begin{split}
y_1(t) &= \mathrm{A}_{11}\tau^h + \mathrm{A}_{12}\tau^{h-1} 
+\cdots+ \mathrm{A}_{1,h+1} + \frac{\mathrm{B}_{11}}{\tau} 
+ \frac{\mathrm{B}_{12}}{\tau^2} + \frac{\mathrm{B}_{13}}{\tau^3} 
+ \cdots
\\
y_2(t) &= \mathrm{A}_{21}\tau^h + \mathrm{A}_{22}\tau^{h-1} 
+\cdots+ \mathrm{A}_{2,h+1} + \frac{\mathrm{B}_{21}}{\tau} 
+ \frac{\mathrm{B}_{22}}{\tau^2} + \frac{\mathrm{B}_{23}}{\tau^3} 
+ \cdots
\\
 & \vdots
\\
y_n(t) &= \mathrm{A}_{n1}\tau^h + \mathrm{A}_{n2}\tau^{h-1} 
+\cdots+ \mathrm{A}_{n,h+1} + \frac{\mathrm{B}_{n1}}{\tau} 
+ \frac{\mathrm{B}_{n2}}{\tau^2} + \frac{\mathrm{B}_{n3}}{\tau^3} 
+ \cdots
\end{split}
\end{equation}
which are all convergent for 
% $t$   %%KWR602: I think clearer to say \tau here, doesn't matter technically since t = \tau^k.  
$\tau$
greater than some constant, where the following hold:
\begin{enumerate}
\item % (a)
For a certain positive integer $k$, $\tau = t^{1/k}$, where the
positive real value of the root is meant;
\item % (b)
The given number $h$ is the highest positive exponent of $\tau$ that
occurs.
\item % (c)
All coefficients $\mathrm{A_{i,j}}$ and $\mathrm{B_{i,j}}$ are well-defined uniquely determined complex numbers;
\item % (d)
Any formal power series $y(t)$ satisfying the formal identity
$g\{y(t),t\}\equiv 0$ and having properties analogous to those of the
series \eqref{Puiseux1} necessarily coincides with one of
the $n$ series above.
\item % (e)
The following formal identity holds:
\[
g(y,t) \equiv \prod_{i=1}^{n} \{y - y_i(t)\}.  
\]
\qed
\end{enumerate}
\hideqed
\end{thm}

Hilbert \cite{Hilbert1892} ascribed the idea of employing Puiseux series to Runge in a work that had appeared
three years earlier, and cited it in a footnote exactly like this.%
\footnote{The original work of Puiseux can be found in
\emph{Liouville's} Journal, vols.~15, 16 (1850, 1851). These
expansions have already been used by \emph{C. Runge} to derive
necessary conditions that an equation between two unknowns have
infinitely many integral solutions. See this Journal, vol.~100,
p.~425.
[Footnote by Hilbert]}
% 
%%KWR906: Next is bracket of Bill's suggestion.  I think here I've finessed around the onus perceived by Bill of changing "g(y,t)" to "g(x,y)" per the intro.  There were good motives for the change to t, and the referees haven't complained about *that* inconsistency.
In order for the idea to get off the ground, however, we need to have at least one series, i.e., $n \geq 1$.  This is where the condition $g(y,t) \in \Zed[y,t] \setminus \Zed[t]$ in Theorem~\ref{thm:goal}, reflecting emendations in Section~\ref{sec:intro}, is used.

Hilbert then notes---and this is the reason to reduce the problem to monic
polynomials---the relation between the elementary symmetric functions
of the roots $y_1,y_2,\dots,y_n$ and the coefficient polynomials
$S_1,S_2,\ldots,S_n$, namely:
\begin{equation}
\label{SymFun} % (2)
\begin{split}
S_1 &= -(y_1 + y_2 +\cdots+ y_n)
\\
S_2 &= (-1)^2(y_1y_2 + y_1y_3 +\cdots+ y_{n-1}y_n)
\\
&\vdots
\\
S_n &= (-1)^n(y_1y_2y_3\cdots y_{n-1}y_n).\\
\end{split}
\end{equation}

%%KWR602: Changed wording here.  Also, my edits above have created a few lines of LaTeX slack.  We could hence insert a few lines saying /why/ the S_j here must become the same as the S_j in the previous section.

The insight is that by Puiseux's theorem, when we plug the expansions
\eqref{Puiseux1} into the symmetric functions \eqref{SymFun},
\emph{the resulting fractional power series for the coefficients all
collapse down to 
% the integral polynomials $S_k$ in~$t$ from \eqref{eyt}}.
integral polynomials in $t$.  And those must be the polynomials $S_k$ in~$t$ from \eqref{gyt}.}
For later reference, it will be convenient to state 
% this simple 
the former observation as a theorem, calling the part of
the expansion with positive exponents the \emph{polynomial part}.

%%KWR531: Maybe adding the words "and only if" will make Referee 3 happier?

\begin{thm} % 10
\label{thm:nub}
For any $g(y,t) \in \Cee[y,t]$ the elementary symmetric functions of
$n$ Puiseux expansions \eqref{Puiseux1} parameterizing the roots for any $t$ collapse down to polynomials
in $\Zed[t]$ if (and only if):
\begin{enumerate}
\item % (a)
The coefficients of 
% all 
the \emph{negative} powers of $\tau$ in the
resulting fractional power series for the coefficients are
\emph{all equal to zero}.
\item % (b)
\label{polypart}
The numerical coefficients of the \emph{``polynomial part"} of the
resulting fractional power series for the coefficients
\emph{are all integers}.
\item % (c)
The numerical coefficients of the\emph{ positive 
fractional powers} of $\tau$ in the resulting fractional power series
for the coefficients are \emph{all equal to zero}.
\qed
\end{enumerate}
\hideqed
\end{thm}

%%KENANDBILL: We changed "x,y" to "y,t" and added the clause "parameterizing the roots for any $t$"

\noindent These three conditions will, with appropriate changes, characterize
the coefficients of any polynomial factor in $\Zed[y,t]$ of $g(y,t)$.
Lemma~\ref{lem:factors} above shows that it suffices to write ``rational
numbers'' instead of ``integers'' in condition~(b).

%%%%%%%%%%%%%%%%%%%%%%%%%%%%%%%%%%%%%%%%%%%%%%%%%%%%%%%%%%%

\section{The formal factors.}
\label{sec:factors}

%%KWR531: I added "nonempty proper" before "subset" to augment the insert of "nontrivial" to address a comment by Referee 3, and added "exclude the two trivial ones" for even more emphasis.

Any nontrivial \emph{formal} polynomial factor of $g(y,t)$ is a polynomial of the
form
\begin{equation}
\label{formal} % (3)
\pi_A(y,t) := \prod_{y_j\in A} (y - y_j),
\end{equation} 
where $A$ is a nonempty proper subset of the roots $\{y_1,y_2,\ldots,y_n\}$. As
Hilbert points out, there are $\binom{n}{2}$ quadratic factors,
$\binom{n}{3}$ cubic factors, $\binom{n}{4}$ quartic factors,
$\binom{n}{5}$ quintic factors, 
% and so on, 
and finally
% $\binom{n}{n-1}$ %%KWR531: Looks ugly depending on where it falls in the line
$n$
factors of degree $n-1$. Additionally, we count the
$n$ linear factors but exclude the two trivial ones, giving a grand total of
\[
\binom{n}{2} + \binom{n}{3} +\cdots+ \binom{n}{n-1} + n = 2^n - 2
\]
possible factors.
So, if $g(y,t)$ is reducible, \emph{some
$\pi_A(y,t)$ must be an integral polynomial factor.} 
% We assign to each set $A$ a unique index $a$ where $a = 1,2,\dots, 2^n - 2$. And by abuse of notation we call $\pi_A(y,t)$ by the name $\pi_a(y,t)$ where $a$ is the unique index assigned to~$A$.
%
Sometimes we prefer to think of $A$ as a single item rather than a set, so we assign it a unique index $a$ where $a = 1,2,\dots, 2^n - 2$. Then
$\pi_a(y,t)$ means the same as $\pi_A(y,t)$.  
These items will become the ``colors'' in the cube lemma.

Let's look at a simple example of a reducible integral polynomial:
\[
g(y,t) := y^3 - t^3.
\]
Then the roots of $g(y,t) = 0$ are $y_1 = t$, $y_2 = \omega t$,
$y_3 = \omega^2 t$ where $\omega^3 = 1$, $\omega \neq 1$.%
\footnote{Moreover, $y_1 = t$, $y_2 = \omega t$, $y_3 = \omega^2 t$
where $\omega^3 = 1$, $\omega \neq 1$ are also the Puiseux expansions
of the roots.} 
When $n = 3$ there are $2^3-2 = 6$ formal factors. Thus the sets $A$
are
\[
\{y_1\}, \quad \{y_2\}, \quad \{y_3\}, \quad 
\{y_1,y_2\}, \quad \{y_1,y_3\}, \quad \{y_2,y_3\},
\]
and we (arbitrarily) assign the indices $a = 1,2,3,4,5,6$ to them,
respectively. Then, by \eqref{formal}, these formal factors are:
\begin{align*}
\pi_{\{y_1\}} &\equiv \pi_1(y,t) = y - y_1 = y - t,
\\
\pi_{\{y_2\}} &\equiv \pi_2(y,t) = y - y_2 = y - \omega t,
\\
\pi_{\{y_3\}} &\equiv \pi_3(y,t) = y - y_3 = y - \omega^2 t,
\\
\pi_{\{y_1,y_2\}} &\equiv \pi_4(y,t) = (y - y_1)(y - y_2) 
= y^2 + \omega ty + \omega^2 t^2,
\\ 
\pi_{\{y_1,y_3\}} &\equiv \pi_5(y,t) = (y - y_1)(y - y_3) 
= y^2 + \omega^2 ty + \omega t^2,
\\ 
\pi_{\{y_2,y_3\}} 
&\equiv \pi_6(y,t) = (y - y_2)(y - y_3) = y^2 +  ty +  t^2.
\end{align*}
We observe that $\pi_1(y,t)$ and $\pi_6(y,t)$ are \emph{integral}
polynomial factors, whereas the other four are not.

%%%%%%%%%%%%%%%%%%%%%%%%%%%%%%%%%%%%%%%%%%%%%%%%%%%%%%%%%%%%%%%%%%%

\section{Using the pigeonhole principle.}
\label{sec:php}

%%KWR531: Tweaked to improve line and page breaks.  I see how you answered the referee's comment.

Our problem now is to discover \emph{at least one} formal factor $\pi_{\alpha}$ that is an
\emph{integral} polynomial.  (We note that its complementary factor is also an integral polynomial.) We begin our search by applying our
hypothesis.

%Let $t_0$ be an integer greater than $C''$, 
%%KWR531: Here is where we get rid of the dangling C''.  Is this all correct?

Take $t_0$ from the hypothesis of Theorem~\ref{thm:goal}
and $\tau_0$ to be the
positive $k$th root of~$t_0$. Usually, $\tau_0$ will be
irrational. By hypothesis, if we substitute $\tau_0$ into all of the
coefficient series of the formal factors $\pi_a(y,t)$, \emph{at least
one of them will be an integral polynomial in $y$}.

Now, along with Hilbert, we observe that if we substitute $2\tau_0$ into all
of the coefficient series of the formal factors $\pi_a(y,t)$, then
\emph{at least one of them will be an integral polynomial in $y$}, by
the assumption that $g(y,2^k t_0)$ is reducible.

Again, if we substitute $3\tau_0$ into all of the coefficient series
of the formal factors $\pi_a(y,t)$ \emph{at least one of them will be
an integral polynomial in~$y$}, by the assumption that $g(y,3^kt_0)$
is reducible. The same is true of $4\tau_0$, $5\tau_0$, and indeed, of
$\sigma\tau_0$ for $\sigma = 1,2,3,\dots$.

Therefore, we obtain \emph{an infinite sequence of integral polynomial
factors $\pi_a(y,\sigma^k t_0)$ in $y$}. Each of them has a unique
index $a \in \{1,2,\dots, 2^n - 2\}$. Let these indices be
$a_1,a_2,a_3,\dots,a_s,\dots$. Then, by the pigeonhole principle,
\emph{at least one index $a_s$ occurs infinitely often}. In our
example above, we can take $a_s = 1$ or $a_s = 6$ and our sequence of
indices contains $1$ or $6$ or both infinitely often. The point
is the following.

% \noindent The corresponding formal polynomial $\pi_{a_{s}}(y,t)$ is a natural
% candidate for our integral polynomial factor.
% \\

\begin{center}
\fbox{\parbox{8cm}{The corresponding formal polynomial $\pi_{a_{s}}(y,t)$ is a natural
candidate for our integral polynomial factor.
}}
\end{center}

%%KWR531: Thought more consistent and better overall to put this in a box too.

\noindent 
To prove that it \emph{is} our integral polynomial factor, we must
verify that its Puiseux 
% series satisfies 
%%KWR602: Previously we had "series satisfy" here; compositor wanted "satisfies" but I think we really mean
expansions satisfy 
the three conditions of
Theorem~\ref{thm:nub}. The rest of Hilbert's paper (and ours) is 
% the proof that \emph{the candidate formal factor $\pi_{a_{s}}(y,t)$ satisfies these three conditions .}
%%KWR531: Now that we've admitted that the formal factor in section 11 can come out different from what we started with, it strikes me we should modulate this.
about pinning down \emph{a formal factor $\pi_{a_{s}}(y,t)$ that satisfies these three conditions.}

%%%%%%%%%%%%%%%%%%%%%%%%%%%%%%%%%%%%%%%%%%%%%%%%%%%%%%%%%%%%%%%%%%%%%%

\section{Framing the cube lemma.}
\label{sec:framing}

Let's consider the first condition of Theorem 16: \emph{The coefficients of all the
\emph{negative} powers of $\tau_0$ in the resulting fractional power
series for the coefficients of $\pi_{a_s}(y,\sigma^kt_0)$ are all
equal to zero}. 
%% KWR602: Important add---is it right?---to satisfy about n versus \nu and the notation.
Without loss of generality we may suppose $a_s$ indexes $\{y_1,\dots,y_{\nu}\}$ where $\nu < n$ is the degree of $\pi_{a_s}$.
Suppose the following system of coefficient power
series for $\pi_{a_s}(y,(\sigma^kt_0))$ has $a_s$ as its index:

\begin{align*}
y_1 + y_2 +\cdots+ y_\nu 
&= \mathrm{C}_{11}(\sigma\tau_0)^h 
+ \mathrm{C}_{12}(\sigma\tau_0)^{h-1} +\cdots+ \mathrm{C}_{1,h+1} 
\\
&\hspace*{5.8em}
+ \frac{\mathrm{D}_{11}}{\sigma\tau_0} 
+ \frac{\mathrm{D}_{12}}{(\sigma\tau_0)^2} 
+ \frac{\mathrm{D}_{13}}{(\sigma\tau_0)^3} +\cdots
\\[-2\jot]
 & \vdots
\\
y_1y_2\cdots y_\nu 
&= \mathrm{C}_{\nu1}(\sigma\tau_0)^{h\nu} 
+ \mathrm{C}_{\nu2}(\sigma\tau_0)^{h\nu-1} 
+\cdots+ \mathrm{C}_{\nu,h\nu+1} 
\\
&\hspace*{6.2em}
+ \frac{\mathrm{D}_{\nu1}}{\sigma\tau_0} 
+ \frac{\mathrm{D}_{\nu2}}{(\sigma\tau_0)^2} 
+ \frac{\mathrm{D}_{\nu3}}{(\sigma\tau_0)^3} +\cdots .
\end{align*}

%%KWR602: I kept "determinate" a little further down but feel "determined as" reads better here.  Also note changes to second sentence and the important query below.
The coefficients $\mathrm{C_{ij}},\mathrm{D_{ij}}$ are all completely
% determinate 
determined as
rational or irrational, real or complex 
% numbers; some of them have the value zero since the positive exponents of $\tau_0$ in general will be smaller than $(n - 1)h$.
numbers.  Some of
them must be zero because the positive exponents of $\tau_0$ in
general will be smaller than $(n - 1)h$.

%%KWR602: Do we have to change "n" to "\nu" in this last line?  

The variable quantity here is the integer $\sigma$. This suggests that
\emph{we rewrite the above fractional series as series in~$\sigma$}
and then obtain
\begin{align*}
y_1 + y_2 +\cdots+ y_\nu
&= C_{11}\sigma^h + C_{12}\sigma^{h-1} +\cdots+ C_{1,h+1} 
\\
&\hspace*{4.1em}
+ \frac{D_{11}}{\sigma} + \frac{D_{12}}{\sigma^2} 
+ \frac{D_{13}}{\sigma^3} +\cdots
\\[-2\jot]
 & \vdots
\\
y_1y_2 \cdots y_\nu
&= C_{\nu1}\sigma^{h\nu} + C_{\nu2}\sigma^{h\nu -1} 
+\cdots+ C_{\nu,h\nu+1}
\\
&\hspace*{4.6em}
+ \frac{D_{\nu1}}{\sigma} 
+ \frac{D_{\nu2}}{\sigma^2} + \frac{D_{\nu3}}{\sigma^3} +\cdots
\end{align*}
where the new coefficients $C_{ij},D_{ij}$ are again determinate numerical
quantities. Suppose that the index of the first occurrence of our
infinitely repeated polynomial factor is $s = \sigma = \mu$. Then
every repetition of the index $\mu$ produces the same $\nu$ power
series in $\sigma$, but with a \emph{larger} value of $\sigma$. Thus,
since there are an infinite number of such indices, \emph{there are infinitely many larger and larger values of $\sigma$ substituted into the power series.
}

%%KWR531: Not sure if/how the referee's comment about "n" versus "nu" is addressed.  We could say, "Hilbert's `$\nu$' is our `$n$' above." 

If we look at the series of \emph{negative} powers for any particular
coefficient, we see that for sufficiently large $\sigma$ it becomes
arbitrarily small in absolute value. Yet, the total power series takes
integral values for all of these values of $\sigma$. That suggests
that the total contribution of the negative powers, for large
$\sigma$, \emph{is an integer of arbitrarily small absolute value,
i.e., zero}.

Thus we might try to argue by contradiction as follows: Assume that
there \emph{are} nonzero coefficients of negative powers and deduce an
absurd conclusion. The possible hitch is that this inference ignores
the ``polynomial part'' of the coefficient series---which could
exactly compensate for a tiny nonzero contribution of the negative
powers.

To show that this is \emph{not} the case we would like to somehow
``eliminate'' the polynomial part of the coefficient series without
losing the property of being an integer for infinitely many values
of~$\sigma$. This suggests \emph{forming suitable linear combinations
of the coefficient series that successively subtract off the
principal terms of the polynomial parts, and leaving finally only
linear combinations of integer-valued series with negative
coefficients}.

To see how this would work, let's look at a typical coefficient
series. Let us choose any of the $\nu$ power series in the system
under consideration, say the power series
\[
\mathcal{P}(\sigma) 
= C_{11}\sigma^{m-1} + C_{12}\sigma^{m-2} +\cdots+ C_{1m} 
+ \frac{D_{11}}{\sigma} + \frac{D_{12}}{\sigma^2} 
+ \frac{D_{13}}{\sigma^3} + \cdots ,
\]
where we have written $m - 1$ for the highest power of~$\sigma$.  

Now comes a new insight. \emph{This is the insight that is key to the
whole proof}. Its simplicity belies the brilliance it took to think of
it. Professional mathematicians are aware of this phenomenon: the
deepest ideas, in the end, are based on a simple observation. The observation can be counterfactual and act as a catalyst.  Here is
Hilbert's: 

%\[    %%KWR531: Got centering by {center} rather than embed inside equation
\begin{center}
\fbox{\parbox{10cm}{Let us suppose that the series $\mathcal{P}(\sigma)$ takes on
integral values, not only for infinitely many values of $\sigma$,
\emph{but also for all the infinitely many values of 
$\sigma + \mu^{(1)}$, where $\mu^{(1)}$ is a fixed increment
independent of~$\sigma$}.}}
\end{center}
%\]

\noindent
To set it in motion, we form the linear combination
\[
\mathcal{P}^{(1)}(\sigma) 
:= \mathcal{P}(\sigma) - \mathcal{P}(\sigma + \mu^{(1)})
\]
and put the polynomial part equal to
\[
\varphi_{m-1}(\sigma) 
:= C_{11}\sigma^{m-1} + C_{12}\sigma^{m-2} +\cdots+ C_{1m}
\]
for brevity. We now start the argument-by-contradiction.

Suppose now that the other coefficients $D_{11},D_{12},D_{13},\dots$
of the power series $\mathcal{P}(\sigma)$ \emph{are not all zero} and
let $D_{1v}/\sigma^v$ be the first term whose coefficient $D_{1v}$
does \emph{not} vanish. Then
\[
\mathcal{P}^{(1)}(\sigma) = \varphi_{m-1}(\sigma) 
- \varphi_{m-1}(\sigma + \mu^{(1)}) 
+ D_{1v} \biggl[ \frac{1}{\sigma^v} 
- \frac{1}{(\sigma + \mu^{(1)})^v} \biggr] + \cdots .
\]
Here the first difference on the right-hand side \emph{is a polynomial
of degree $m - 2$ in~$\sigma$}; we put
\[
\varphi_{m-2}(\sigma) 
= \varphi_{m-1}(\sigma) - \varphi_{m-1}(\sigma + \mu^{(1)}).
\]
We expand the remaining terms on the right-hand side in decreasing 
powers of $\sigma$; then we obtain%
\footnote{The details, with simplified notation, are:
\begin{align*}
D_{1v} \biggl[ \frac{1}{\sigma^\nu} 
- \frac{1}{(\sigma + \mu)^\nu} \biggr]
&= \frac{D_{1v}}{\sigma^\nu} \biggl[
1 - \frac{1}{(1 + \frac{\mu}{\sigma})^{\nu}} \biggr]
\\
&= \frac{D_{1v}}{\sigma^\nu} 
\biggl[ 1 - \biggl\{ 1 + \binom{-\nu}{1} \frac{\mu}{\sigma} 
+ \binom{-\nu}{2} \biggl( \frac{\mu}{\sigma} \biggr)^2 
+ \cdots \biggr\} \biggr]
\\
&= \frac{D_{1v}}{\sigma^\nu} 
 \biggl[ \frac{\mu\nu}{\sigma} - \frac{\nu(\nu + 1)}{2}
\biggl( \frac{\mu}{\sigma} \biggr)^2 + \cdots \biggr]
\\
&= \frac{\mu\nu D_{1v}}{\sigma^{\nu+1}} \biggl[
1 - \frac{\nu + 1}{2} \biggl( \frac{\mu}{\sigma} \biggr)
+ \cdots \biggr],
\end{align*}
and this last expression on the right-hand side is equivalent to that
in the main body of the paper.}
% end of footnote
\[
\mathcal{P}^{(1)}(\sigma) 
= \varphi_{m-2}(\sigma) + \mu^{(1)}v\frac{D_{1v}}{\sigma^{v+1}} 
+\cdots \,.
\]
\emph{We have reduced the maximum degree of the polynomial part by one
unit}. Moreover, $\mathcal{P}^{(1)}(\sigma)$ \emph{takes on integral
values for infinitely many~$\sigma$}.

We have \emph{not} proved that such an increment $\mu^{(1)}$ exists,
but if we could, then we could make a first step in reaching our goal
of producing a series of negative powers of $\sigma$ that 
% is 
%%KWR531: I see you changed "which" to "that" but the change made me realize that the referent of "an integer" is "series".  So I changed the verb.
evaluates to
an integer for infinitely many values of~$\sigma$.

To carry out a similar program to \emph{reduce the maximum degree of
the polynomial part to $m - 3$, then to $m - 4$, and so on until
finally to zero}, we would have to have to prove the existence of $m$
fixed increments $\mu^{(k)}$, $k = 1,2,\ldots,m$ \emph{whose values
are all independent of $\sigma$ and such that if we substitute any of
the integers}
\[
\boxed{\begin{smallmatrix} %% usual {matrix} will not fit
\mu &  &  &  & 
\\
\mu + \mu^{(1)} &  &  &  & 
\\
\mu + \mu^{(2)} & \mu + \mu^{(1)} + \mu^{(2)} & & & 
\\
\mu + \mu^{(3)} & \mu + \mu^{(1)} + \mu^{(3)}
& \mu + \mu^{(2)} + \mu^{(3)} 
& \mu + \mu^{(1)} + \mu^{(2)} + \mu^{(3)} & 
\\
\vdots & \vdots & \vdots & \vdots \hspace*{2.5em} \cdots 
& \ddots \hfill 
\\
\mu + \mu^{(m)} & \mu + \mu^{(1)} + \mu^{(m)}
& \mu + \mu^{(2)} + \mu^{(m)} & \cdots \qquad \cdots
& \mu + \mu^{(1)} +\cdots+ \mu^{(m)}
\end{smallmatrix}}
\]
\emph{for $\sigma$ in $\mathcal{P}(\sigma)$, the values will all be
integers}.

Note that the proof of the existence of the increment $\mu^{(1)}$
amounts to proving that $a_{\mu}$ and $a_{\mu + \mu^{(1)}}$ are both
\emph{the same number in the set of indices~$a_s$}. A similar property
holds for the set of all the above sums of fixed increments
$\mu^{(k)}$ (for $k = 1,2,\dots,m$) with $\mu$, namely that they all
are the subscripts of \emph{the same number in the set of
indices~$a_s$}. In our running example they are all equal to~$1$ or
they are all equal to~$6$.

The proof of the existence of these increments \emph{is the content of the cube lemma}.
% \emph{Hilbert cube lemma}. 
% We now state it using his formulas much as he visualized them in his paper.
%%KWR531: Given the additions above, I think this is no longer where "we now state" it.  Hence change.
% My edit also caused LaTeX not to split the theorem across pages, which I think is desirable though it adds a spillover page to our paper.
To serve the context of Hilbert's proof, we restate it using his formulas much as he visualized them in his paper.

\begin{thm} % 11
Let $a_1,a_2,a_3,\ldots$ be an infinite sequence in which the general
term, $a_s$, is one of the $a$ positive numbers $1,2,\ldots,a$.
Moreover, let $m$ be any positive integer. Then we can always find $m$
positive integers $\mu^{(1)},\mu^{(2)},\ldots,\mu^{(m)}$ such that for
infinitely many integers $\mu$ the $2^m$ elements
\[
\setlength{\arraycolsep}{3pt} %% narrower column spacing
\boxed{\begin{matrix}
a_ \mu  &  &  &  & 
\\
a_{\mu + \mu^{(1)}} &  &  &  &
\\
a_{\mu + \mu^{(2)}} & a_{\mu + \mu^{(1)} + \mu^{(2)}}  &  &  & 
\\
a_{\mu + \mu^{(3)}} & a_{\mu + \mu^{(1)} + \mu^{(3)}} 
& a_{\mu + \mu^{(2)} + \mu^{(3)}}
& a_{\mu + \mu^{(1)} + \mu^{(2)} + \mu^{(3)}} & 
\\
\vdots & \vdots & \vdots & \enspace\vdots \hspace*{2.4em} \cdots 
& \ddots \hfill 
\\
a_{\mu + \mu^{(m)}} & a_{\mu + \mu^{(1)} + \mu^{(m)}}
& a_{\mu + \mu^{(2)} + \mu^{(m)}} & \cdots \qquad \cdots
& a_{\mu + \mu^{(1)} + \mu^{(2)} +\cdots+ \mu^{(m)}} 
\end{matrix}}
\]
are all equal to the same number $G$, where $G$ is one of the numbers
$1,2,\dots,a$.
\end{thm}

Thus we see that the statement \emph{arises naturally from the
necessity of proving that the coefficients of the negative powers of
$\sigma$ must be all equal to zero}. It is the strengthened form of
the pigeonhole principle we mentioned earlier. It is stronger because
it imposes a \emph{structure} on the distribution of infinitely many
common values $a_s$ whereas the pigeonhole principle only implies
their \emph{existence}.

%%%%%%%%%%%%%%%%%%%%%%%%%%%%%%%%%%%%%%%%%%%%%%%%%%%%%%%%%%%%%%%%%%%%%%

\section{The Coefficients of the Negative Powers of $\sigma$ are Zero.}
\label{sec:neg}

Employing the idea above, we form the following $m$ linear
combinations:
\begin{align*}
\mathcal{P}^{(1)}(\sigma)
&= \mathcal{P}(\sigma) - \mathcal{P}(\sigma + \mu^{(1)}),
\\
\mathcal{P}^{(2)}(\sigma)
&= \mathcal{P}^{(1)}(\sigma) - \mathcal{P}^{(1)}(\sigma + \mu^{(2)}),
\\[-\jot]
 \qquad &\qquad \vdots
\\
\mathcal{P}^{(m)}(\sigma) 
&= \mathcal{P}^{(m-1)}(\sigma) 
- \mathcal{P}^{(m-1)}(\sigma + \mu^{(m)}). 
\end{align*}
It follows from what we proved earlier that each of these $m$ power
series also assumes integer values for infinitely many integral
arguments $\sigma = \mu$.

As we indicated, assuming the cube lemma, we obtain
\[
\mathcal{P}^{(2)}(\sigma) = \varphi_{m-3}(\sigma) 
+ \mu^{(1)}\mu^{(2)}v(v + 1) \frac{D_{1v}}{\sigma^{v+2}} + \cdots \,,
\]
where $\varphi_{m-3}(\sigma)$ is a polynomial in $\sigma$ of degree
$m - 3$. \emph{After $m$ steps we finally arrive at the formula}
\[
\mathcal{P}^{(m)}(\sigma) 
= \mu^{(1)} \mu^{(2)} \cdots \mu^{(m)} v(v + 1) \cdots
(v + m - 1) \frac{D_{1v}}{\sigma^{v+m}} + \cdots \,.
\]

Since this power series begins with negative powers of $\sigma$, we
can find a positive number $\Gamma$ such that for all values of
$\sigma$ that exceed $\Gamma$ the absolute value of the power series
\emph{will be smaller than one}. On the other hand, the power series
$\mathcal{P}^{(m)}(\sigma)$ is itself \emph{equal to an integer for
infinitely many arguments} $\sigma$ and since an integer whose
absolute value is less than one is necessarily equal to zero, it
follows that \emph{there are infinitely many integers $\sigma$ for
which the power series vanishes}.

But, our last formula shows us that 
\[
\lim_{\sigma\to\infty}
\left[ \sigma^{v+m}\mathcal{P}^{(m)}(\sigma) \right]
= \mu^{(1)} \mu^{(2)} \cdots \mu^{(m)} v(v + 1) \cdots
(v + m - 1) D_{1v},
\]
where the expression on the right-hand side represents a quantity
\emph{different from zero}. This last result stands in
\emph{contradiction} with the conclusion proved above, and therefore
\emph{it is impossible that a nonzero coefficient $D_{1v}$ occurs
among the coefficients $D_{11},D_{12},D_{13},\dots$}. It follows in
the same way that \emph{also the coefficients
$D_{2i},D_{3i},D_{4i},\dots,D_{\nu i}$ must all be equal to zero}.

This completes the proof of the first condition of Theorem~\ref{thm:nub} about the Puiseux
expansions of the coefficients of $\pi_{a_s}(y,t)$.
\qed

\medskip

This step was the heart of Hilbert's proof and his paper's most
brilliant insight. The other parts are clever too, but in our opinion
this best shows his penetrating originality.

%%%%%%%%%%%%%%%%%%%%%%%%%%%%%%%%%%%%%%%%%%%%%%%%%%%%%%%%%%%%%%%%%%%%%%

\section{The Coefficients of the Polynomial Part are Rational Numbers.}
\label{sec:rat}

The next condition of Theorem~\ref{thm:nub} to be verified is:
\emph{the numerical coefficients in the polynomial part of the Puiseux
expansions of the coefficients of $\pi_{a_s}(y,t)$ are rational
numbers}. Our expansion has collapsed to the polynomial part, i.e. 
\begin{equation}
\label{poly}
\mathcal{P}(\sigma) 
= C_{11}\sigma^{m-1} + C_{12}\sigma^{m-2} +\cdots+ C_{1m},
\end{equation}
where the right-hand side assumes integer values for infinitely many
values of $\sigma$. If we set the right-hand side equal to these
integers for $m$ values of $\sigma$, we obtain $m$ linear equations
in $m$ unknowns $C_{11},C_{1,2},\dots,C_{1m}$ which have a
\emph{rational solution} by Cramer's rule. 
% This finishes the proof of the condition---we do not have to replace ``rational'' with ``integral'' because Proposition~\ref{prop:Que2Zed} showed that ``rational'' is enough.
%%KWR531: Tweaked to achieve page-saving break.
By Proposition~\ref{prop:Que2Zed}, getting ``rational'' suffices to prove the condition.
\qed

%%%%%%%%%%%%%%%%%%%%%%%%%%%%%%%%%%%%%%%%%%%%%%%%%%%%%%%%%%%%%%%%%%

\section{Only Integral Powers of \lowercase{$t$}\,.}
\label{sec:int}

The final condition of Theorem~\ref{thm:nub} to be verified is:
\emph{the only nonzero terms in the polynomial part of the Puiseux
expansions of the coefficients of $\pi_{a_s}(y,t)$ are those with
integral powers of~$t$}.

%Take $\tau_0$ to be a \emph{prime} number $p$ larger than $C''$ and recall $\sigma\tau_0=\tau.$
%%
%We now determine $2^n - 2$ distinct prime numbers
%$p',p'',\dots,p^{2^n-2}$ all greater than $p$. Then also for each of
%these prime numbers there exists at least one among the $2^n - 2$
%formal factors whose coefficients have the above polynomial form
%\eqref{poly}. However, since the number of prime numbers
%$p,p',p'',\dots,p^{2^n-2}$ is equal to $2^n - 1$ while the number
%formal factors only reaches $2^n - 2$, necessarily \emph{there must
%exist at least one formal factor admitting a double representation by
%these polynomials} \eqref{poly}. That is to say, as above:

%%KWR531: I don't see any place where the prime numbers are "determined", so I think simply "select" was meant.  
%% I made some other word changes---please check that I preserved correctness.
%% Finally, I put p' inside parentheses when raising to a power.

Take $\tau_0$ to be a \emph{prime} number larger than $t_0$ in the statement of 
Theorem~\ref{thm:goal} and recall that $\sigma\tau_0 = \tau.$
We now select $2^n - 2$ distinct prime numbers $p_{\ell}$ all greater than $\tau_0$.
For each of $\tau_0$ and the $p_{\ell}$, at least one among the $2^n - 2$
formal factors has coefficients in the above polynomial form
\eqref{poly}. However, since we have $2^n - 1$ primes and
only $2^n - 2$ formal factors, \emph{there must
exist at least one formal factor admitting a double representation by
these polynomials} \eqref{poly}. That is to say, there are distinct primes $p,p' > t_0$ such that
\begin{align*}
y_1 + y_2 +\cdots+ y_\nu
&= A_{11} p^{-(m-1)/k} \tau^{m-1} 
+ A_{12}p^{-(m-2)/k} \tau^{m-2} +\cdots+ A_{1m}
\\   
 & \vdots
\\
y_1 y_2 \cdots y_{\nu}
&= A_{\nu1} p^{-(m-1)/k} \tau^{m-1} 
+ A_{\nu2} p^{-(m-2)/k} \tau^{m-2} +\cdots+ A_{\nu m},
\\
\intertext{and simultaneously,}
y_1 + y_2 +\cdots+ y_\nu
&= A'_{11} (p')^{-(m-1)/k} \tau^{m-1}
+ A'_{12} (p')^{-(m-2)/k} \tau^{m-2} +\cdots+ A'_{1m}
\\   
 & \vdots
\\
y_1 y_2 \cdots y_\nu
&= A'_{\nu1} (p')^{-(m-1)/k} \tau^{m-1}
+ A'_{\nu2} (p')^{-(m-2)/k} \tau^{m-2} +\cdots+ A'_{\nu m} \,.
\end{align*}    

\noindent 
Since, by Puiseux's theorem, the coefficients of the powers of $\tau$ are unique, if we equate coefficients of equal powers of $\tau$ on the right-hand sides we obtain:
\[
\begin{aligned}
A_{11}p^{-(m-1)/k}  
&= A'_{11} (p')^{-(m-1)/k}, & \cdots \quad A_{\nu 1} p^{-(m-1)/k} &= A'_{\nu 1} (p')^{-(m-1)/k} 
\\
A_{12}p^{-(m-2)/k}  
&= A'_{12} (p')^{-(m-2)/k}, & \cdots \quad A_{\nu 2} p^{-(m-2)/k} &= A'_{\nu 2} (p')^{-(m-2)/k} 
\\
 &\vdots  &  &\vdots\quad \quad & \quad
\\
A_{1m} &= A'_{1m}, &  \quad A_{\nu m} &= A'_{\nu m} \, .
\end{aligned}
\]

% Because the $A$ and $A'$ coefficients are all rational
%numbers and $p$ and $p'$ are distinct prime numbers, the
%equations above show us that \emph{the only coefficients that can be
%different from zero are those for which the corresponding exponent of
%$\tau$ must be an integer divisible by~$k$}.

%%BILLANDKEN: We've already addressed what the referee said explicitly by the new t_1 usage when "t" is a constant.  Below we fill in one more point of reasoning whose absence maybe contributed to the referee's not following.  It is a slightly more subtle point than sqrt{2} being irrational, anyway.  It still bridges to the words in italics which you had next.

\noindent
The final numerical point is that for any rational number $r$ that is not an integer, and distinct primes $p$ and $p'$, the ratio $p^r /p'^r$ is irrational.  Suppose the ratio were equal to the rational number $c/d$ in lowest terms, and let us also put $r = k/\ell$ in lowest terms.  Cross-multiplying then gives 
\[
d^{\ell} p^k = c^{\ell} p'^k.
\]
For this to happen, $p$ must divide $c$.  Let $a$ be the largest integer such that $p^a$ divides $c$.  Then we must have $\ell = ak$ in order to cancel out all factors of $p$.  This means that $\ell/k$ must be an integer, contradicting the condition on $r$.  
%
%Thus if $\ell$ is any of the numerators involving  $m$ in exponents above and we have 
%\[
%A p^{\ell/k} = A' p'^{\ell/k},
%\]
%the exponent $\ell/k$ must be an integer.
  
We can apply this where $\ell$ is a numerator of one of the exponents of $p$ and $p'$, such as $\ell = m-1$, because we established in Section~\ref{sec:rat} that the coefficients $A_{i,j}$ of the polynomial part are all rational.  We conclude that \emph{the only coefficients that can be different from zero are those for which the corresponding exponent of
$\tau$ must be an integer divisible by~$k$}.

It follows that the power
series of our system \emph{are polynomials in $\tau^k$ with rational
coefficients}, and if we put $\tau^k = t$, we obtain
\begin{align*}
y_1 + y_2 +\cdots+ y_\nu &= F_1(t)
\\[-\jot]  
\vdots\qquad &
\\
y_1 y_2 \cdots y_\nu &= F_\nu(t),    
\end{align*}
where $F_1(t),\dots,F_{\nu}(t)$ are \emph{polynomials in $t$
with rational coefficients}.

%%KWR531: Hemmed and hawed about preserving the colon before this last ensemble of equations.

This completes the proof of the third condition of Theorem~\ref{thm:nub} of the Puiseux
expansions of the coefficients of some formal factor $\pi_{a_s}(y,t)$.
\qed

\bigskip
We note that the final formal factor $\pi_{a_s}(y,t)$ is not necessarily the same one that we started with.
% We have shown that a suitable formal factor $\pi_{a_s}$ 
All we needed was that it fulfills the three
conditions of Theorem~\ref{thm:nub}, and therefore the proof of Theorem~\ref{thm:goal} is complete.
\qed

%%%%%%%%%%%%%%%%%%%%%%%%%%%%%%%%%%%%%%%%%%%%%%%%%%%%%%%%%%%%%%%%%%%
%% KWR 11/3: Applied same clarification about "cube lemma of irred.
%% theorem?" to section title

\section{Later Proofs of the Irreducibility Theorem.}
\label{sec:later}
After Hilbert, many mathematicians offered other proofs of the
irreducibility theorem.

Most of the modern proofs of the (two-variable) irreducibility
theorem are based on that of Karl D\"orge~\cite{Dorge}, which sharpened an
idea of Thoralf Skolem~\cite{skolem}. 
D\"orge proved it without using the
cube lemma and obtained a stronger result.  To begin contrasting his and Hilbert's results, recall Hilbert's
statement that if $f\in\Zed[x,y_1,\ldots,y_s]$ is irreducible,
then for infinity many
$t_1,\ldots,t_s \in \Zed$, $f(x,t_1,\ldots,t_s)$ is irreducible as a member of $\Zed[x]$.

Now let $|f|$ be the maximum of $8$ and the absolute values of the coefficients of $f$.
(The reason for insisting $|f| \ge 8$ is technical.)
A simplified statement of D\"orge's theorem is the following.

%\begin{thm}\label{thm:Dorge}
%There is a function $c(d,s)$ such that the following holds.
%Let $f\in\Zed[x,t_1,\ldots,t_s]$ be irreducible of degree $d$.
%Let $N>|f|^{c(d,s)}$.
%Then the number of $(a_1,\ldots,a_s)\in \{-N,\ldots,N\}^s$ such that 
%$f(x,a_1,\ldots,a_s)$ is not irreducible is at most $|f|^{c(d,s)}N^{s-(1/2)}\log N$.
%\end{thm}

\begin{thm}\label{thm:Dorge}
There is a function $c(d,s)$ such that the following holds.
Let $f\in\Zed[x,y_1,\ldots,y_s]$ be irreducible of degree $d$.
Let $N>|f|^{c(d,s)}$.
Then the number of $(t_1,\ldots,t_s)\in \{-N,\ldots,N\}^s$ such that 
$f(x,t_1,\ldots,t_s)$ is not irreducible is at most $|f|^{c(d,s)}N^{s-(1/2)}\log N$.
\end{thm}

%\noindent
%Note that the number of such $(a_1,\ldots,a_s)$ has density 0.
%D\"orge actually presented a generalization of this theorem where he replaces
%$\Zed$ with the integers of a finite extension of a number field.

\noindent
Note that the number of such $(t_1,\ldots,t_s)$ has density 0.
D\"orge actually presented a generalization of this theorem where he replaces
$\Zed$ with the integers of a finite extension of a number field.

%D\"orge also showed (in fact, this was his primary interest) that if $f$, viewed
%as an element of $\Zed[t_1,\ldots,t_s][x]$, has 
%Galois group $G$,
%then the number of $(a_1,\ldots,a_s)\in \{-N,\ldots,N\}^s$ such that 
%$f(x,a_1,\ldots,a_s)$ does not have Galois group $G$ is at most
%$|f|^{c(d,s)}N^{s-(1/2)}\log N$. 
%And again, he actually presented a generalization of this theorem 
%that replaces
%$\Zed$ with the integers of a finite extension of a number field.

D\"orge also showed (in fact, this was his primary interest) that if $f$, viewed
as an element of $\Zed[y_1,\ldots,y_s][x]$, has 
Galois group $G$,
then the number of $(t_1,\ldots,t_s)\in \{-N,\ldots,N\}^s$ such that 
$f(x,t_1,\ldots,t_s)$ does not have Galois group $G$ is at most
$|f|^{c(d,s)}N^{s-(1/2)}\log N$. 
And again, he actually presented a generalization of this theorem 
that replaces
$\Zed$ with the integers of a finite extension of a number field.

Lang~\cite{diogeom,diogeomfund} and
Prasolov~\cite{polynomials} have expositions of D\"orge's proof.
Franz~\cite{Franz} also gave a proof that does not use the cube lemma,
and this is expounded further by Schinzel~\cite{polynomialssel}. There
is a another proof by Fried~\cite{fried}. 
Serre~\cite{serre} recasts these results in geometric terms and 
presents results about which groups can be Galois groups.

%%%%%%%%%%%%%%%%%%%%%%%%%%%%%%%%%%%%%%%%%%%%%%%

\section{Conclusion: Hilbert's World.}
\label{sec:conc}

We have shown the significance of the cube lemma in the context of
Hilbert's original paper. We return to the questions of how Hilbert might have
expanded it in the direction of Ramsey theory and why he didn't.  That Hilbert was
the world's master in the relationship between number theory and logic until G\"odel emerged, 
while Ramsey was motivated by a problem in logic, gives more reason to ask.
We close with a speculative answer: The world in which
Hilbert was immersed was and remains at least one level of exponentiation higher than the ground floor of Ramsey theory.

Recall from Section~\ref{sec:cube-lemma} that
we defined the ``Hilbert Cube Numbers'' $H(m,c)$ to be the least $H$ such that
every $c$-coloring of $\{1,\dots,H\}$ has a monochromatic $m$-cube.
%%KWR906: Actually we said "every H-length interval" for convenience in the
%%inductive proof---can we expect the reader to get the equivalence?
The best known upper and lower bounds appear still to be those of
Gunderson and R\"odl~\cite{GR}:
\[
c^{(1 - \epsilon_c)(2^m - 1)/m} \leq H(m,c) \leq (2c)^{2^{m-1}},
\]
where $\epsilon_c \to 0$ as $c \to \infty$.\footnote{%
The same upper bound was
recently ascribed to \cite{GRS2} by Conlon, Fox, and Sudakov
\cite{CFS} but see also S\'andor~\cite{sandor} with different
asymptotics.
} %% end of footnote
This establishes doubly-exponential growth of $H(m,c)$ in $m$ for any fixed $c$, in
contrast to the singly-exponential growth of the Ramsey numbers, in particular
\[
R_2(m) \leq \binom{2m-2}{m-1} \leq \frac{4^m}{\sqrt{m}}.
\]
%I find this at http://math.mit.edu/~fox/MAT307-lecture06.pdf
Moreover, while Erd\H{o}s and Tur\'an~\cite{et} proved that $H(2,c)$
is asymptotic to $c^2$, much less is known about $H(m,c)$ for fixed $m \geq 3$ 
(see remarks in \cite{BCEG} which appear still in force), and in
\cite{CFS} it is noted that $H(m,2)$ depends on unknown properties of van
der Waerden numbers.  This all puts $H(m,c)$ on a higher and harder plane
than analogous cases of $R_c(m)$ in the Ramsey world.  So what world was Hilbert in?

%%KWR531: Which was its fatal attraction to me starting in mid-1997 when I fatefully visited Ernst Mayr in Munich before Complexity'97 in Ulm.  The extra exponential is exactly what makes this world---including Mulmuley's world---capable of surmounting the "Natural Proofs" barrier (and "Algebrization" too).  

%The question of whether Hilbert might have
%expanded on it bids comparison with Ramsey's motivation
%in~\cite{ramsey30}. That was a problem in logic not number theory
%\emph{per se}. But Hilbert was the world's master in the relationship
%between number theory and logic until G\"odel emerged, so one may ask
%again why Hilbert didn't pursue this further or develop areas of
%extremal combinatorics \emph{vis-\`a-vis} logic in the direction of
%Ramsey theory. We close with a speculative answer: The world in which
%Hilbert was immersed 
% is as different from that of Ramsey theory as ``doubly-exponential'' is from ``singly-exponential.''
%was and remains at least one level of exponentiation higher than the ground of Ramsey theory.

The years 1890--1893 saw the publication of Hilbert's great
foundational works in commutative algebra, including his basis theorem
and \textit{Nullstellensatz} \cite{Hilbert1890,Hilbert1893}. A common
thread is the notion of \emph{regularity}: given
a finitely-specified system of elements that may have arbitrarily
large values of some parameter $t$ (such as the degrees of certain polynomials
over a ring), there is some integer $t_0$ such that, for all 
$t \geq t_0$, the system conforms to a simple description. Hilbert
first proved his basis theorem nonconstructively. Later was it shown
that the growth of the relevant $t_0$ (in terms of the degrees $d$ of
basis elements or the $n$-variable equations in the
\textit{Nullstellensatz}) is doubly-exponential, of order up to
$d^{2^n}$.  We could try to equate the growth of $H(m,c)$ with 
that of Ramsey numbers by regarding the cube as a graph of size $M = 2^m$ and saying
$H(m,c)$ is ``singly exponential in $M$.'' But $m$, not $M$, is still the
natural parameter, just as with $n$, not $2^n$, in the \textit{Nullstellensatz}.

%%KWR531: I changed this again

Hence our answer is simply that Hilbert was occupied with
more rarefied levels of algebra and analysis where nonconstructive methods were often more salient
than double-exponential effective ones. Irreducibility of polynomials plays into irreducible
varieties and primary decompositions of polynomial ideals, which
Hilbert's student Emanuel Lasker (the world chess champion) and
% prot\'eg\'e 
colleague 
Emmy Noether built upon for some great work in the next
two decades that became more algorithmic. 
%
%%KWR531: My use of "protege" is reflected by Wikipedia's words in her bio:
%% "In 1915, she was invited by David Hilbert and Felix Klein to join the mathematics department at the University of Göttingen, a world-renowned center of mathematical research. The philosophical faculty objected, however, and she spent four years lecturing under Hilbert's name."
%% I wanted to list Hilbert's student first and connect Noether to Hilbert, but I guess "protege" came off as sexist.
%
In the meantime, Hilbert swooped down to the
ground-level task of formalizing Euclid's geometry in the later 1890s,
which presaged his work on formal logic.

The divide in purpose and growth rate does not ward us off from
appreciating the cube numbers and seeking other uses for them. That is
why we have devoted this paper to expounding their original use and
context. We have highlighted how the cube lemma completed an insight
about estimates by infinite series. We hope that our exposition will
foster a greater appreciation of combinatorial underpinnings of more
``analytical'' areas of mathematics.
\\
\\
\textbf{Acknowledgements}.
We thank Joseph P. Varilly for formatting help and the referees for helpful comments that greatly improved the exposition and clarity.

%%BILLANDKEN: OK to ack. Varilly like that?  Anything else?

\end{document}